\documentclass[a4paper,12pt]{article}

\usepackage[english]{babel}
\usepackage[utf8]{inputenc}
\usepackage{amsmath, amsthm, amsfonts}
\usepackage{graphicx}
\usepackage{fullpage}
\usepackage{setspace}
\usepackage{relsize}
\usepackage[round]{natbib}
\theoremstyle{plain}
\newtheorem{theorem}{Theorem}[section]
\newtheorem{lemma}[theorem]{Lemma}

\newtheorem{corollary}[theorem]{Corollary}
\theoremstyle{definition}

\newcommand*\colvec[3][]{
    \begin{pmatrix}\ifx\relax#1\relax\else#1\\\fi#2\\#3\end{pmatrix}
}

\newcommand{\ba}{\mathbf{a}}
\newcommand{\bb}{\mathbf{b}}
\newcommand{\bc}{\mathbf{c}}
\newcommand{\bd}{\mathbf{d}}
\newcommand{\sigmathetasq}{\sigma_x^2\cos^2\theta+2\rho\sigma_x\sigma_y\cos\theta\sin\theta+\sigma_y^2\sin^2\theta}

\title{An Analysis of Dual-Issue Final-Offer Arbitration}

\author{Brian Powers}

\date{\today}

\begin{document}
\maketitle

\doublespacing

\begin{abstract}
We discuss final-offer arbitration where two quantitative issues are in dispute and model it as a zero-sum game. Under reasonable assumptions we both derive a pure strategy pair and show that it is both a local equilibrium and furthermore that it is the unique global equilibrium. 
\end{abstract}

\section{Introduction}
Should negotiating parties fail to arrive at an agreeable solution, \textit{arbitration} serves as a mechanism whereby a binding resolution may be reached. In \textit{conventional arbitration (CA)}, the disputing parties submit their cases to an agreed upon arbiter who has full power to craft whatever fair and just settlement he sees fit. It is widely accepted, however, that CA has a number of undesirable properties, in particular what has been called the ``chilling effect'': since both parties know the arbiter will craft a compromise, they tend to take extreme positions. Since it is commonly held that a settlement reached through negotiation is preferable to a settlement reached through arbitration, one can view the purpose of a compulsory arbitration as motivating the parties to reach an agreement during negotiations. This is the paradox of arbitration: the best arbitration mechanism is that which is used least often.

It was Stevens~\citeyearpar{stevens1966} who suggested a simple arbitration mechanism now known as \textit{Final-Offer Arbitration (FOA)}. In FOA, the arbiter must select one of the \textit{final offers} submitted by the parties and has no prerogative to craft a compromise settlement. The theory was that such uncertainty in the final outcome would combat this chilling effect driving the two parties to make final offers that are ``close'' to one another, or better still motivate them to reach agreement during negotiations. 

Since 1975 when FOA was adopted by Major League Baseball for salary disputes, variants of FOA have been used in various states in public sectors where labor does not have the right to strike (e.g. police, firefighters). A growing body of literature has been developed by legal scholars, economists and game theorists studying both the theoretical and empirical properties of FOA.

The first theoretical model of FOA was introduced by Crawford~\citeyearpar{crawford1979}. With the assumption that both parties know with certainty the arbiter's opinion of a ``fair'' settlement, he showed that FOA would inevitably lead to the same outcome as conventional arbitration. Farber~\citeyearpar{farber1980}, Chatterjee~\citeyearpar{chatterjee1981}, and Brams and Merrill~\citeyearpar{brams1983} independently developed game theoretic models of single-issue FOA for which players are uncertain of the arbiter's behavior. Farber studied the effect of risk aversion by one of the parties, and derived the strategy pair which in many cases is a Nash equilibrium. Chatterjee and Brams and Merrill model the game as zero-sum and consequently assumed both parties are risk-neutral. Brams and Merrill provide sufficient conditions for the existence of a pure equilibrium.
In all three models, the arbiter is assumed by the players to choose a ``fair'' settlement from a probability distribution commonly known to both players and select whichever player's offer is closest in absolute value. Kilgour~\citeyearpar{kilgour1994} studied the game theoretic properties of FOA and extended the Brams-Merrill model to allow for risk-aversion on the part of the players. Dickinson~\citeyearpar{dickinson2006} further showed that optimism on the part of the players, in the form of a biased prior distribution, drives the final-offers apart.

If multiple issues are in dispute, FOA has been primarily implemented in two ways \citep{stern1975}. Under \textit{Issue-by-Issue FOA (IBIFOA)}, the arbiter may craft a compromise of sorts from the two parties' offers by choosing some components from one and some from the other. Alternatively, \textit{Whole Package FOA (WPFOA)} requires that the arbiter select one offer in its entirety. A multi-issue model of FOA was first discussed by Crawford~\citeyearpar{crawford1979} and further developed by Wittman~\citeyearpar{wittman1986}. Here the main concern was the existence of a Nash equilibrium under various assumptions. Wittman was also able to show in his model that increased risk-aversion leads a player to make a less extreme final-offer.  Olson~\citeyearpar{olson1992} discussed how the single-issue model does not accurately reflect arbiter behavior when more than one issue is in dispute. 

In his initial paper introducing FOA, Stevens cautions against the use of the ``Whole Package'' variant, stating that ``such a system would run the danger of generating unworkable awards...the arbitration authority might be forced to choose between two extreme positions, each of which was unworkable''\citep{stevens1966}. Tulis~\citeyearpar{tulis2013} elaborates: ``One common criticism of package final-offer arbitration is that parties may be tempted to include outrageous offers.'' He further claims that ``issue-by-issue final offers...are more aligned with the objectives of final-offer arbitration.'' We argue the opposite - that both players' optimal strategy in a multiple-issue FOA is to make all final-offers reasonable. Furthermore, the additional variance in the awards from WP, as opposed to IBI, acts as a greater motivator for the parties to reach agreement during negotiations. We show this by extending the model of Brams and Merrill to multiple-issues and proceed to explicitly construct a pure strategy pair, proving it is the unique optimal strategy pair.

\section{Dual-Issue Final-Offer Arbitration}
Our model extends the model defined by Brams and Merrill~\citeyearpar{brams1983}. Let Player I be the minimizer and Player II the maximizer in this zero-sum game. Let us consider the case where each player makes not a single valued offer, but an ordered pair $(x_i, y_i)$, $i=1,2$. For this model, we assume that the two issues in dispute are quantitative and valued identically by the players. An example of such a situation is one in which wage and workers compensation amounts are in dispute; workers' compensation may be valued at the \emph{expected} compensation amount (in the probabilistic sense). Even issues which are not monetary, such as number of sick days, may have a straightforward monetary valuation by the parties. We will assume that these issues are positively correlated across the industry. Let us further assume that both players are restricted to a strategy space $S$ which is an arbitrarily large, \emph{compact} subset of $\mathbb{R}_2$. Both players are uncertain of the arbiter's opinion of a fair settlement $(\xi,\eta)$, but assume that the arbiter (or a fact-finder) is sampling from relevant industry data to form an opinion. Thus, by the Central Limit Theorem, we suppose that their common prior distribution for $(\xi,\eta)$ is a bivariate normal distribution, $N(\mathbf{\mu},\mathbf{\Sigma})$ and it is common knowledge (Aumann). Let us assume without loss of generality that $\mathbf{\mu}=\mathbf{0}$. Also let 
$$\mathbf{\Sigma}=\left[
\begin{array}{cc}
\sigma_x^2& \rho\sigma_x\sigma_y\\
\rho\sigma_x\sigma_y & \sigma_y^2
\end{array}\right],$$ where $\rho > 0$. 

In the multi-issue case, FOA is typically handled in one of two ways: Issue-by-Issue (IBI) or Whole-Package (WP). Under IBIFOA the arbiter rules independently on each issue presented. A compromise of sorts may be crafted in this way. If the arbiter uses the IBI mechanic,  the players are engaged in two independently decided single-issue FOA games. By the Brams-Merrill Theorem~\citeyearpar{brams1983}, we know that the unique optimal strategy pair of the players is given by
\begin{equation}
(x_1^*,y_1^*)=\left(-\frac{\sigma_x\sqrt{2\pi}}{2},-\frac{\sigma_y\sqrt{2\pi}}{2} \right) \quad (x_2^*,y_2^*)=\left(\frac{\sigma_x\sqrt{2\pi}}{2},\frac{\sigma_y\sqrt{2\pi}}{2} \right). \label{IBIFOAsol}
\end{equation}

Under WPFOA the arbiter must rule in favor of one final-offer vector in its entirety. It is in this variant that the choice of a distance criterion needs to be chosen by the arbiter. The ``distance'' from a final-offer point $(x_i,y_i)$ to $(\xi,\eta)$ may be determined in a number of ways\footnote{Many other distance concepts are reasonable and worthy of consideration. These include absolute total difference, any $L_p$ metric, Mahalanobis distance, or standardized distance. These are considered by the author in detail elsewhere.}. For this model we assume it is common knowledge that the arbiter uses Euclidean distance, or an $L_2$ norm:
\begin{equation}D_{L_2}\big((x,y),(\xi,\eta)\big)=\sqrt{(\xi-x)^2+(\eta-y)^2}.\end{equation}

\section{Properties of Dual-Issue FOA under $L_2$ Distance}
We now establish some properties of the game. Suppose Player I chooses pure strategy $\ba=(x_1,x_2)$ and Player II chooses pure strategy $\bb=(x_2,y_2)$, and the arbiter considers $(\xi,\eta)$ a fair settlement. We define $C_i(\ba,\bb)$,  as the set of points in $\mathbb{R}_2$ which are strictly closer to Player $i$'s final-offer than to the other player's, namely
\begin{equation}\label{eqC1}
C_1(\ba,\bb) := \left \{ (x,y) : (x_1-x)^2 + (y_1-y)^2 < (x_2-x)^2 + (y_2-y)^2 \right\},
\end{equation}
\begin{equation}\label{eqC2}
C_2(\ba,\bb) := \left \{ (x,y) : (x_1-x)^2 + (y_1-y)^2 > (x_2-x)^2 + (y_2-y)^2 \right\}.
\end{equation}
It is immediately apparent that $C_1(\ba,\bb) = C_2(\bb,\ba)$.
The midset is
\begin{equation}\label{eqMid}
Mid(\ba,\bb) := \left \{ (x,y) : (x_1-x)^2 + (y_1-y)^2 = (x_2-x)^2 + (y_2-y)^2 \right\}.
\end{equation}
We observe that if $\ba \not=\bb$ then $Mid(\ba,\bb)$ is a line so $P\big( (\xi,\eta) \in Mid(\ba,\bb)\big)=0$. We can now define the expected payoff to Player II from I
\begin{equation}\label{eqPayoff}
K(\ba,\bb) = 
\begin{cases} 
x_1+y_1 & \ba=\bb\\
(x_1+y_1)P\big( (\xi,\eta)\in C_1(\ba,\bb) \big) + (x_2+y_2)P\big( (\xi,\eta) \in C_2(\ba,\bb)\big) & \ba \not=\bb
\end{cases}
\end{equation}

The first property is \textbf{anonymity of final-offers}; the arbiter essentially does not care which player submits which final-offer.
\begin{lemma}\label{lemAnonymous}
$K( \ba,\bb) = K(\bb, \ba)$.
\end{lemma}
\begin{proof}
If $\ba=\bb$ the proof is trivial. Assume $\ba \not= \bb$.
\begin{alignat*}{2}
K(\ba,\bb) &= (x_1+y_1)P\big( (\xi,\eta)\in C_1(\ba,\bb) \big) + (x_2+y_2)P\big( (\xi,\eta) \in C_2(\ba,\bb)\big)  \\
&= (x_1+y_1)P\big( (\xi,\eta)\in C_2(\bb,\ba) \big) + (x_2+y_2)P\big( (\xi,\eta) \in C_1(\bb,\ba)\big)\\
&= (x_2+y_2)P\big( (\xi,\eta) \in C_1(\bb,\ba)\big) + (x_1+y_1)P\big( (\xi,\eta)\in C_2(\bb,\ba) \big)\\
&= K(\bb,\ba)
\end{alignat*}
\end{proof}
The next property is due to the symmetry of the bivariate normal distribution about $(0,0)$.
\begin{lemma} \label{lemNegativeK}
Let $-\ba = (-x_1,-y_1)$ and $-\bb=(-x_2,-y_2)$. Then $K(-\ba,-\bb) = -K(\ba,\bb)$.
\end{lemma}
\begin{proof}
This proof makes use of two facts: First, $(\xi,\eta)\in C_i(\ba,\bb) \Leftrightarrow (-\xi,-\eta) \in C_i(-\ba,-\bb)$, $i=1,2$. Secondly, $(\xi,\eta)$ and $(-\xi,-\eta)$ follow the same probability distribution.
\begin{alignat*}{2}
K(-\ba,-\bb) &= (-x_1-y_1)P\big( (\xi,\eta)\in C_1(-\ba,-\bb) \big) + (-x_2+-y_2)P\big( (\xi,\eta) \in C_2(-\ba,-\bb)\big)  \\
&= -\Big((x_1+y_1)P\big( (\xi,\eta)\in C_1(-\ba,-\bb) \big) + (x_2+y_2)P\big( (\xi,\eta) \in C_2(-\ba,-\bb)\big) \Big)  \\
&= -\Big((x_1+y_1)P\big( (-\xi,-\eta)\in C_1(\ba,\bb) \big) + (x_2+y_2)P\big( (-\xi,-\eta) \in C_2(\ba,\bb)\big) \Big)  \\
&= -\Big((x_1+y_1)P\big( (\xi,\eta)\in C_1(\ba,\bb) \big) + (x_2+y_2)P\big( (\xi,\eta) \in C_2(\ba,\bb)\big) \Big)  \\
&= -K(\ba,\bb)
\end{alignat*}
\end{proof}
Next we show that if the players play opposite pure strategies, the expected payoff of the game is zero.
\begin{lemma}\label{lemOppositePure}
Let $-\bb=(-x_2,-y_2)$. Then $K(-\bb,\bb)=0$.
\end{lemma}
\begin{proof}
This proof also relies on the fact that $(\xi,\eta)$ and $(-\xi,-\eta)$ follow the same probability distribution.
\begin{alignat*}{2}
K(-\bb,\bb) &= (-x_2,-y_2)P\big( (\xi,\eta)\in C_1(-\bb,\bb) \big) + (x_2+y_2)P\big( (\xi,\eta) \in C_2(-\bb,\bb)\big)  \\
&= (x_2+y_2)\Big( P\big( (\xi,\eta) \in C_2(-\bb,\bb)\big) - P\big( (\xi,\eta)\in C_1(-\bb,\bb) \big)\Big)  \\
&= (x_2+y_2)\Big( P\big( (\xi,\eta) \in C_2(-\bb,\bb)\big) - P\big( (-\xi,-\eta)\in C_1(\bb,-\bb) \big)\Big)  \\
&= (x_2+y_2)\Big( P\big( (\xi,\eta) \in C_2(-\bb,\bb)\big) - P\big( (-\xi,-\eta)\in C_2(-\bb,\bb) \big)\Big)  \\
&=(x_2+y_2)\Big( P\big( (\xi,\eta) \in C_2(-\bb,\bb)\big) - P\big( (\xi,\eta)\in C_2(-\bb,\bb) \big)\Big)\\
&= 0
\end{alignat*}
\end{proof}

With the previous lemmas, we can show that the value of the game is zero.
\begin{lemma}\label{valuezero}
Consider a bivariate FOA game where the arbiter chooses {$(\xi, \eta)\sim N(\mathbf{0},\mathbf{\Sigma})$} as a fair settlement and uses $L_2$ distance to measure closeness. The value of the zero-sum game is zero.
\end{lemma}
\begin{proof}
Because the strategy space $S$ of each player is compact, by the general minimax theorem the game has a value $v$. 

First suppose an optimal pure strategy pair $\ba^*,\bb^*$ exists. Suppose $v>0$. Then for any pure strategy $\ba$ of Player I, $K(\ba, \bb^*)\ge v>0$. But by Lemma~\ref{lemOppositePure} 
$K(-\bb^*,\bb^*) = 0$, contradicting that $v>0$. Similarly it cannot be the case that $v<0$. Therefore $v=0$.

Now suppose that optimal mixed strategies $F_1^*,F_2^*$ exist. Suppose $v>0$. Then for any mixed strategy $F_1$, 
\begin{equation}\label{vgzero}
K(F_1, F_2^*)\ge v > 0.\end{equation}
Player II may approximate the optimal strategy $F_2^*$ by $\hat{F}_2^*$ where probability mass is concentrated only on a finite symmetric subset $T\subset S$ such that for $\epsilon>0$ small enough and for any mixed strategy $F_1$,
\begin{equation}\label{vgzero2}
K(F_1, \hat{F}_2^*)\ge v-\epsilon > 0.\footnote{This claim follows from three arguments: First, by a well known  theorem of Varadharajan the space of all probability measures  $M(S)$ is a compact metric space in weak topology. Secondly, the set of probability all measures under weak topology  concentrated on finite subsets of a compact metric space $S$  are themselves dense in the space of all probability measures on $S$. Lastly, by a well known theorem of Prohorov, any compact subset $T$ of $M(S)$ is characterized by 
the property that given $\delta$ positive, there exists a compact subset of $C$ of $T$ such that 
$\mu (C) >1-\delta$ for all $\mu$ in the set $S$.}
\end{equation}

Define $$g_1^*(x,y) = \hat{f}_2^*(-x,-y), \forall (x,y)\in T$$ and call the associated mixed strategy $G_1^*$. 

\begin{alignat*}{2}
K(G_1^*,\hat{F}_2^*) 
& =\sum\displaylimits_{(\ba,\bb) \in T\times T}
g_1^*(\ba) \hat{f}_2^*(\bb) K(\ba,\bb)\\
&=\sum\displaylimits_{(\ba,\bb) \in T\times T} \hat{f}_2^*(-\ba) g_1^*(-\bb) K(\ba,\bb)\\
&=\sum\displaylimits_{(\ba,\bb) \in T\times T} \hat{f}_2^*(-\ba) g_1^*(-\bb)K(\bb,\ba)\\
&=-\sum\displaylimits_{(\ba,\bb) \in T\times T} g_1^*(-\bb) \hat{f}_2^*(-\ba) K(-\bb,-\ba)
\end{alignat*}
With a change of variables $\bc=-\bb, \bd=-\ba$,
\begin{alignat*}{2}
&=-\sum\displaylimits_{(\bc,\bd) \in T\times T} g_1^*(\bc) \hat{f}_2^*(\bd) K(\bc,\bd)\\
&=-K(G_1^*,\hat{F}_2^*) \end{alignat*}
Therefore $K(G_1^*,\hat{F}_2^*)=0$, contradicting (\ref{vgzero}), so $v\le 0$. In a similar manner we can show that $v \ge 0$.
\end{proof}

The next property states that optimal pure strategy pairs, if they exist, must be symmetric about the origin.
\begin{lemma}
Consider a bivariate Final-Offer Arbitration game where the arbiter chooses $(\xi, \eta)\sim N(\mathbf{0},\mathbf{\Sigma})$ as a fair settlement, and uses $L_2$ distance\footnote{This lemma is true for any $L_p$ metric, but for simplicity the proof is provided only for $L_2$.} for deciding closeness.  Then $(x_2,y_2)$ is an optimal pure strategy for Player II if and only if $(-x_2,-y_2)$ is an optimal pure strategy for Player I. \label{lemSymPureStrat}
\end{lemma}
\begin{proof}
Suppose $\bb^*=(x_2^*,y_2^*)$ is an optimal pure strategy for Player II. Because the value of the game is zero, 
\begin{equation}\label{bstaroptimal}K(\ba,\bb^*)\ge 0, \forall \ba.\end{equation}
If $-\bb^*$ is not an optimal pure strategy for Player I then there exists $\bb^\circ$ such that
$$K(-\bb^*,\bb^\circ) >0.$$
By Lemmas~\ref{lemAnonymous} and \ref{lemNegativeK},
\begin{alignat*}{2}
K(-\bb^\circ,\bb^*) &= K(\bb^*,-\bb^\circ)\\
&= -K(-\bb^*,\bb^\circ)\\
&<0
\end{alignat*}
but this contradicts (\ref{bstaroptimal}), so it must be the case that $-\bb^*$ is an optimal pure strategy for Player I. The converse of the lemma is shown in an analogous way.
\end{proof}

Recall from (\ref{eqPayoff}), that if Player I chooses $\ba=(x_1,y_1)$ and Player II chooses $\bb=(x_2,y_2)$, assuming $\ba \not=\bb$, 
\begin{alignat*}{2}
K(\ba,\bb) &= (x_1+y_1)P\big((\xi,\eta)\in C_1(\ba,\bb)\big) + (x_2 + y_2)P\big((\xi,\eta)\in C_2(\ba,\bb)\big)\\
&=(x_1+y_1)P\big((\xi,\eta)\in C_1(\ba,\bb)\big) + (x_2 + y_2)[1-\big((\xi,\eta)\in C_1(\ba,\bb)\big)]\\
&=(x_2+y_2) + (x_1+y_1-x_2-y_2)P\big((\xi,\eta)\in C_1(\ba,\bb)\big).
\end{alignat*}
\begin{lemma}
 Suppose in the bivariate FOA game as described the arbiter chooses $(\xi, \eta)\sim N(\mathbf{0},\mathbf{\Sigma})$ and uses the $L_2$ metric to measure closeness. If a pure optimal strategy pair $\ba^*=(x_1^*,y_1^*),  \bb^*=(x_2^*,y_2^*)$ exists, then $x_2^*\ge 0,y_2^* \ge 0$ and $x_1^*\le0, y_1^*\le 0$. \label{lemPositive}
 \end{lemma}
 \begin{proof}
We know that if both players are playing optimally then the expected payoff is zero. Suppose only one of Player II's offers is negative\footnote{Player II cannot possibly be playing optimally if both $x_2^*<0$ and $y_2^*<0$, for in this case Player I may simply agree to the Player II's final offer and happily accept a negative net settlement.}; WLOG let $x_2^*< 0$. By playing $(-x_2^*,-y_2^*)$, Player I is guaranteeing a zero expected payoff. Suppose Player I instead switches to $(x_2^*,-y_2^*)$. 
If $y_2^*=0$ then the final offers are identical and the net award is $x_2^*$. Therefore let us assume $y_2^*>0$.
\begin{alignat*}{2}
K\big( (x_2^*,-y_2^*), (x_2^*,y_2^*)\big ) &= (x_2^*+y_2^*) + (x_2^*-y_2^*-x_2^*-y_2^*)P\big ((\xi,\eta) \in C_1( (x_2^*,-y_2^*), (x_2^*,y_2^*)  )\big)\\
&=x_2^* + y_2^*\Big(1-2P\big ((\xi,\eta) \in C_1( (x_2^*,-y_2^*), (x_2^*,y_2^*)  )\big)\Big)
\end{alignat*} Since $C_1=\left\{ (x,y) : y<0\right\}$, $P((\xi,\eta)\in C_1) =P(\eta <0)=\frac{1}{2}$. Therefore, $$K\big( (x_2^*,-y_2^*), (x_2^*,y_2^*)\big )=x_2^*<0.$$ This contradicts that $(x_2^*,y_2^*)$ is an optimal pure strategy for Player II. Thus $x_2^*\ge 0$. Because the choice of component is arbitrary, $y_2^*\ge0$ as well.  
The argument is the same to show that Player I's component offers must be non-positive in order to play optimally.
 \end{proof}

\section{Local Optimality of Pure Strategies}
Having established some of the properties of the game in question, we now derive a pure strategy pair for the players and show that it is a local equilibrium.  

\begin{theorem}
If the arbiter chooses $(\xi, \eta)\sim N(\mathbf{0},\mathbf{\Sigma})$ and uses the $L_2$ metric to measure closeness, the solution points for the two players $i=1,2$
\begin{equation}
(x_i^*,y_i^*)= \left((-1)^i\frac{\sqrt{2\pi(\sigma_x^2 + 2\rho\sigma_x\sigma_y+\sigma_y^2)}}{4},(-1)^i\frac{\sqrt{2\pi(\sigma_x^2 + 2\rho\sigma_x\sigma_y+\sigma_y^2)}}{4}\right)
\end{equation}
constitute a local equilibrium provided $\rho>\max\left\{
-\frac{\sigma_x^2+3\sigma_y^2}{4\sigma_x\sigma_y}, 
-\frac{3\sigma_x^2+\sigma_y^2}{4\sigma_x\sigma_y}
\right\}.$
\end{theorem}
\begin{proof}
Recall that 
\begin{equation}
K(\ba,\bb) =(x_2+y_2) + (x_1+y_1-x_2-y_2)P(\text{Player I wins})
\end{equation}
The event that ``Player I wins'' occurs precisely when the arbiter picks a random fair settlement $(\xi,\eta)$ and
\begin{equation}(x_1-\xi)^2+(y_1-\eta)^2 < (x_2-\xi)^2+(y_2-\eta)^2\end{equation}
which is equivalent to
\begin{equation}(x_2-x_1)\xi + (y_2-y_1)\eta < \frac{x_2^2+y_2^2-x_1^2-y_1^2}{2}=w.\end{equation}
Letting $\Omega=(x_2-x_1)\xi + (y_2-y_1)\eta$, we have that $\Omega\sim N(0,\sigma_\Omega^2)$ where
\begin{equation}
\sigma^2_\Omega = (x_2-x_1)^2\sigma_x^2+2(x_2-x_1)(y_2-y_1)\rho\sigma_x\sigma_y+(y_2-y_1)^2\sigma_y^2.
\end{equation}
And $\Omega/\sigma_\Omega$ follows a standard normal distribution.
Thus we may express the expected payoff as 
\begin{equation}
K(\ba,\bb)  =(x_2+y_2) + (x_1+y_1-x_2-y_2)\Phi(z)\end{equation}
where $\Phi(z)$ is the distribution function of a standard normal random variable and \begin{equation}z=\frac{w}{\sigma_\Omega}=\frac{x_2^2+y_2^2-x_1^2-y_1^2}{2\sqrt{(x_2-x_1)^2\sigma_x^2+2(x_2-x_1)(y_2-y_1)\rho\sigma_x\sigma_y+(y_2-y_1)^2\sigma_y^2}}.\label{equationz}
\end{equation}

%
%

The four partial first derivatives of $K(\ba,\bb)$ are then
\begin{alignat}{2}
\frac{\partial K}{\partial x_1} &= \Phi(z)+ (x_1+y_1-x_2-y_2)\phi(z)\left(-\frac{x_1}{\sigma_\Omega}+\frac{(x_2-x_1)\sigma_x^2+(y_2-y_1)\rho\sigma_x\sigma_y}{\sigma_\Omega^2}z\right) \label{l2dx1}\\
\frac{\partial K}{\partial y_1} &= \Phi(z) +(x_1+y_1-x_2-y_2)\phi(z)\left(-\frac{y_1}{\sigma_\Omega}+\frac{(x_2-x_1)\rho\sigma_x\sigma_y+(y_2-y_1)\sigma_y^2}{\sigma_\Omega^2}z\right)\label{l2dy1}\\
\frac{\partial K}{\partial x_2} &= 1-\Phi(z) +(x_1+y_1-x_2-y_2)\phi(z)\left(\frac{x_2}{\sigma_\Omega}-\frac{(x_2-x_1)\sigma_x^2+(y_2-y_1)\rho\sigma_x\sigma_y}{\sigma_\Omega^2}z\right)\label{l2dx2}\\
\frac{\partial K}{\partial y_2} &= 1-\Phi(z) +(x_1+y_1-x_2-y_2)\phi(z)\left(\frac{y_2}{\sigma_\Omega}-\frac{(x_2-x_1)\rho\sigma_x\sigma_y+(y_2-y_1)\sigma_y^2}{\sigma_\Omega^2}z\right)\label{l2dy2}
\end{alignat}

If the players have optimal pure strategies $\ba^*$ and $\bb^*$ then we must have all four first derivatives zero. By setting them equal to zero at $(\ba^*,\bb^*)$ and by adding all (\ref{l2dx1}) - (\ref{l2dy2}) we get
\begin{equation}0=2-\frac{(x_1^*+y_1^*-x_2^*-y_2^*)^2}{\sigma_\Omega^*}\phi(z^*) \label{l2sumeq}\end{equation}
so  $x_1^*+y_1^*-x_2^*-y_2^*\not=0$. By adding (\ref{l2dx1}) and (\ref{l2dx2}) we have
\begin{equation}
0 = 1 +\frac{x_2^*-x_1^*}{\sigma_\Omega^*}(x_1^*+y_1^*-x_2^*-y_2^*)\phi(z^*) \label{sumdx1dx2}
\end{equation}
and by adding (\ref{l2dy1}) and (\ref{l2dy2}) we have
\begin{equation}
0 = 1 +\frac{y_2^*-y_1^*}{\sigma_\Omega^*}(x_1^*+y_1^*-x_2^*-y_2^*)\phi(z^*).\label{sumdy1dy2}
\end{equation}
From (\ref{sumdx1dx2}), (\ref{sumdy1dy2}) and Lemma~\ref{lemPositive} we know that \begin{equation}x_2^*-x_1^*=y_2^*-y_1^*=d^*>0.\label{thm25eq1}\end{equation} Note also that
\begin{equation}
\sigma_\Omega^{*2}=d^{*2}(\alpha+\beta),
\end{equation}
where $\alpha=\sigma_x^2+\rho\sigma_x\sigma_y$ and $\beta=\rho\sigma_x\sigma_y+\sigma_y^2$. Furthermore, we now have that
\begin{equation}\label{zstar}
z^*=\frac{d^*((x_2^*+x_1^*)+(y_2^*+y_1^*))}{2\sigma_\Omega^{*}}=\frac{x_2^*+x_1^*+y_2^*+y_1^*}{2\sqrt{\alpha+\beta}}.
\end{equation}

We may now simplify the four equations derived from (\ref{l2dx1})-(\ref{l2dy2}) as
\begin{alignat}{2}
0 
& = \Phi(z^*)+ 2\phi(z^*)\left(\frac{x_1^*}{\sqrt{\alpha+\beta}}-\frac{\alpha}{\alpha+\beta}z^*\right)
\label{l2dx1star}\\
0 
&= \Phi(z^*) + 2\phi(z^*)\left(\frac{y_1^*}{\sqrt{\alpha+\beta}}-\frac{\beta}{\alpha+\beta}z^*\right)
\label{l2dy1star}\\
0 
&= 1-\Phi(z^*) - 2\phi(z^*)\left(\frac{x_2^*}{\sqrt{\alpha+\beta}}-\frac{\alpha}{\alpha+\beta}z^*\right)
\label{l2dx2star}\\
0 
&= 1-\Phi(z^*) - 2\phi(z^*)\left(\frac{y_2^*}{\sqrt{\alpha+\beta}}-\frac{\beta}{\alpha+\beta}z^*\right)
\label{l2dy2star}
\end{alignat}
By taking $(\ref{l2dx1star})+(\ref{l2dy1star})-(\ref{l2dx2star})-(\ref{l2dy2star})$ and using (\ref{zstar}) we get
\begin{alignat*}{2}
0 &= -2 + 4\Phi(z^*) + 2\phi(z^*)\left(\frac{x_1^*+y_1^*+x_2^*+y_2^*}{\sqrt{\alpha+\beta}} -2\frac{\alpha+\beta}{\alpha+\beta}z^*\right)\\
\frac{1}{2}&=\Phi(z^*)+2\phi(z^*)(2z^*-2z^*)\\
&=\Phi(z^*)
\end{alignat*}
Thus $z^*=0$, and by simplifying the four equations (\ref{l2dx1star})-(\ref{l2dy2star}) we get that $x_1^*=y_1^*=-x_2^*=-y_2^*$. We simplify $\sigma^{*2}_\Omega = 4x_1^{*2}(\alpha+\beta)$ and noting that $\phi(0)=\frac{1}{\sqrt{2\pi}}$, equation~(\ref{l2dx1star}) becomes 
\begin{equation}
 0= \frac{1}{2}+2\phi(0)\left(\frac{x_1^*}{\sqrt{\alpha+\beta}}\right),
 \end{equation}
 or equivalently
 \begin{equation}
 x_1^* = -\frac{\sqrt{2\pi(\sigma_x^2+2\rho\sigma_x\sigma_y+\sigma_y^2)}}{4}.
\end{equation}

To show that $\bb^*$ is a local maximum for Player II and $\ba^*$ is a local minimum for Player~I we look at the second partial derivatives evaluated at $(\ba^*,\bb^*)$. Letting $$v(\ba,\bb)=(x_1+y_1-x_2-y_2), \quad u(\ba,\bb)=(x_2-x_1)\sigma_x^2+(y_2-y_1)\rho\sigma_x\sigma_y$$ 
$$\frac{\partial K}{\partial x_1} = \Phi(z)+ \phi(z)\left(-\frac{vx_1}{\sigma_\Omega}+\frac{vuw}{\sigma_\Omega^3}\right)$$
\begin{alignat*}{2}
 \frac{\partial^2K}{\partial x_1^2} &= \phi(z)\left(-\frac{x_1}{\sigma_\Omega}+\frac{uz}{\sigma_\Omega^2}\right) +\phi(z)\left(-z\right)\frac{dz}{dx_1}\left(-\frac{vx_1}{\sigma_\Omega}+\frac{vuw}{\sigma_\Omega^3}\right)\\
 &+\phi(z)\left(\left[-\frac{x_1}{\sigma_\Omega}
 - \frac{v}{\sigma_\Omega}
 -\frac{vx_1u}{\sigma_\Omega^3}\right]
  + \left[
  \frac{uw}{\sigma_\Omega^3} - \frac{vw\sigma_x^2}{\sigma_\Omega^3} - \frac{vux_1}{\sigma_\Omega^3} + \frac{\frac{3}{2}vx_1w}{\sigma_\Omega^5}
  \right]
 \right)\\
\left.\frac{\partial^2K}{\partial x_1^2}\right|_{(\ba^*,\bb^*)} &= \phi(0)\left(-\frac{2x_1^*}{\sigma_\Omega^*} -\frac{v^*}{\sigma_\Omega^*} - \frac{2u^*v^*x_1^*}{\sigma_\Omega^{*3}}\right)\\
 &=\phi(0)\left(-\frac{2x_1^*}{-2x_1^*\sqrt{\alpha+\beta}} -\frac{4x_1^*}{-2x_1^*\sqrt{\alpha+\beta}} - \frac{8(-2x_1^*\alpha)x_1^{*3}}{-8x_1^{*3}(\alpha+\beta)^{3/2}}\right)\\
 &=\frac{\phi(0)}{\sqrt{\alpha+\beta}}\left(3-\frac{2\alpha}{\alpha+\beta}\right)
\end{alignat*}

since $\sigma_\Omega^* = 2 (-x_1^*)\sqrt{\alpha+\beta} $. 
This will be positive if and only if $\alpha+3\beta>0$, 
or equivalently, 
\begin{equation}\rho > -\frac{\sigma_x^2+3\sigma_y^2}{4\sigma_x\sigma_y}.\label{l2xxpos}\end{equation}

Similarly, 
$$\left.\frac{\partial^2 K}{\partial y_1^2}\right|_{(\ba^*,\bb^*)} = \frac{\phi(0)}{\sqrt{\alpha+\beta}}\left(
    3 -\frac{2\beta}{\alpha+\beta}
 \right)$$
 which is positive if and only if 
 \begin{equation}
\rho>-\frac{3\sigma_x^2+\sigma_y^2}{4\sigma_x\sigma_y}\label{l2yypos}.\end{equation}
 Note that it is impossible for both (\ref{l2xxpos}) and (\ref{l2yypos}) to be unsatisfied, as this would imply that $\alpha<0$ and $\beta<0$ and thus $\alpha+\beta<0$.
 
 Letting $t(\ba,\bb)= (x_2-x_1)\rho\sigma_x\sigma_y + (y_2-y_1)\sigma_y^2 $, the mixed partial derivative is
\begin{alignat*}{2}
\frac{\partial^2 K}{\partial y_1 \partial x_1} &=
\frac{\partial}{\partial y_1} \frac{\partial K}{\partial x_1}\\
&=\frac{\partial}{\partial y_1} \left( \Phi(z)+ \phi(z)\left(-\frac{vx_1}{\sigma_\Omega}+\frac{vuw}{\sigma_\Omega^3}\right) \right)\\
&=\phi(z)\left(-\frac{y_1}{\sigma_\Omega}+\frac{tz}{\sigma_\Omega^2}\right)+\phi(z)\left(-z\right)\frac{dz}{dy_1}\left(-\frac{vx_1}{\sigma_\Omega}+\frac{vuw}{\sigma_\Omega^3}\right)\\
  &+\phi(z)\left(\frac{-x_1}{\sigma_\Omega} - \frac{vx_1t}{\sigma_\Omega^3} 
    -\frac{\rho\sigma_x\sigma_y vw}{\sigma_\Omega^3} 
    -\frac{y_1uv}{\sigma_\Omega^3}
    +\frac{3utwv 
  }{\sigma_\Omega^5}
  +  \frac{u w}{\sigma_\Omega^3}
 \right)\\
 \left.\frac{\partial^2 K}{\partial y_1 \partial x_1} \right|_{(\ba^*,\bb^*)}&=\phi(0)\left(-\frac{x_1^*}{\sigma_\Omega^*}-\frac{x_1^*}{\sigma_\Omega^*}-\frac{x_1^*v^*t^*}{\sigma_\Omega^{*3}}-\frac{x_1^*u^*v^*}{\sigma_\Omega^{*3}} \right)\\
&=\frac{x_1^*}{\sigma_\Omega^{*3}}\phi(0)\left(-2\sigma_\Omega^{*2}-v^*(u^*+t^*) \right)\\
&=\frac{x_1^*}{\sigma_\Omega^{*3}}\phi(0)\left(-2(4x_1^{*2}(\alpha+\beta) )-4x_1^*(-2x_1^*(\alpha+\beta)) \right)\\
&=\frac{x_1^*}{\sigma_\Omega^{*3}}\phi(0)\left(0 \right)\\
 &=0
\end{alignat*}

Then $K_{x_1x_1}K_{y_1y_1}-K_{x_1y_1}^2>0$ as long as 
$$\rho > \max\left\{
-\frac{\sigma_x^2+3\sigma_y^2}{4\sigma_x\sigma_y}, 
-\frac{3\sigma_x^2+\sigma_y^2}{4\sigma_x\sigma_y}
\right\}.$$
It can be similarly verified that $\bb^*$ is a local maximum for Player II when Player I plays $\ba^*$, with the same condition on $\rho$.
\end{proof}

\section{Global Optimality of Pure Strategies}
We now proceed to show that the pure strategies found in the preceding section are indeed globally optimal and thus represent the unique optimal strategy pair. We first briefly consider the special case where $\sigma_x=\sigma_y$.
\begin{theorem}
If $\sigma_x=\sigma_y$, then the pure strategy pair $\ba^*,\bb^*$ is a global equilibrium.\label{thmL2GE5}
\end{theorem}
\begin{proof}
Suppose Player II plays $\bb^*=(x_2^*,x_2^*)$. If Player I fixes $\tilde{w}<2x_2^*$ and selects a pure strategy $\ba=(x,y)$ with $x+y=\tilde{w}$, then he will wish to choose $x$ to minimize $$K(\ba,\bb^*)=2x_2^*+(\tilde{w}-2x_2^*)\Phi(z).$$ Since $\tilde{w}-2x_2^*<0$, this is equivalent to maximizing $z$. Equation (\ref{equationz}) becomes
$$z=\frac{2x_2^{*2}-x^2-(\tilde{w}-x)^2}{2\sigma_x\sqrt{(x_2^*-x)^2+2\rho(x_2^*-x)(x_2^*-\tilde{w}+x)+(x_2^*-\tilde{w}+x)^2}}.$$
The numerator, $-2x^2+2\tilde{w}x+2x_2^{*2}-\tilde{w}^2$, is maximized when $x=\frac{\tilde{w}}{2}$, while the function in the denominator under the radical is minimized when $x=\frac{\tilde{w}}{2}$. Therefore, if Player II chooses the pure strategy $\bb^*$, it is sub-optimal for Player I to play any pure strategy off the line $y=x$. Since in the one-dimensional case the strategies $(\ba^*,\bb^*)$ are a global pure equilibrium (see \cite{brams1983}), the proof is complete. 
\end{proof}

Before showing that $\ba^*, \bb^*$ is a global pure strategy equilibrium for $\sigma_x \not=\sigma_y$, we have to establish a few lemmas. 

\begin{lemma}
Suppose Player II chooses strategy $\bb^*=(x_2^*,x_2^*)$. If Player I selects pure strategy $\ba=(x_1,y_1)$ then $z(\ba,\bb^*)=0$ iff $(x_1,y_1)$ lies on the circle of radius $\sqrt{2}x_2^*$ centered at the origin. Furthermore, $x_1^2+y_1^2<2x_2^{*2}$ iff $z>0$ and  $x_1^2+y_1^2>2x_2^{*2}$ iff $z<0$.\label{lemL2GE4}
\end{lemma}
\begin{proof}
From (\ref{equationz})
$$z(\ba,\bb^*)=\frac{2x_2^{*2}-(x_1^2+y_1^2)}{2\sigma_\Omega},$$ and the proof is straightforward.
\end{proof}

\begin{lemma}
If another pure strategy $\ba=(x_1,y_1) \not=-\bb^*=(-x_2^*,-x_2^*)$ exists such that $K(\ba,\bb^*)\le 0$, then $x_1+y_1< 0$ and either $$
x_1^2+y_1^2<2x_2^{*2}
\quad \text{ or } \quad
x_1+y_1\le-2x_2^*.$$\label{lemL2GE3}
\end{lemma}
\begin{proof}
Suppose $x_1+y_1\ge0$. Because the net offer of Player II, $2x_2^*>0$, and Player II has a positive  probability $p$ of being chosen by the arbiter, the expected payoff $$K(\ba,\bb^*)=p(2x_2^*)+(1-p)(x_1+y_1)>0.$$This contradicts our assumption.\\
Suppose $x_1^2+y_1^2\ge 2x_2^{*2}$. Then $z\le 0$ by Lemma~\ref{lemL2GE4} and $\Phi(z)\le\frac{1}{2}$. Suppose also that $x_1+y_1>-2x_2^*$. Then
$$x_1+y_1-2x_2^* = -4x_2^*+\epsilon$$ for some $0<\epsilon<2x_2^*$. But then 
\begin{alignat*}{2}
K(\ba,\bb^*) &= 2x_2^* + (x_1+y_1-2x_2^*)\Phi(z)\\
&= 2x_2^* -(4x_2^*-\epsilon)\Phi(z)\\
&\ge 2x_2^* -(4x_2^*-\epsilon)\frac{1}{2}\\
&= \frac{\epsilon}{2}\\
&>0
\end{alignat*}
and this contradicts our assumption.
\end{proof}

The following general lemma, a special case of which was used in Theorem~\ref{thmL2GE5} will be needed subsequently.
\begin{lemma}
Let $$f(x)=\frac{g(x)}{h(x)}$$ where both $g(x)$ and $h(x)$ are continuously differentiable functions. Suppose $g(x)$ has a unique global maximum at $x_g$ and $h(x)$ has a unique global minimum at $x_h$ (and neither function has any other local extrema). Then $f$ is maximized at some point between $x_g$ and $x_h$. \label{lemMaxFracFunction}
\end{lemma}
\begin{proof}
Let $\bar{x}=\max\{x_g,x_h\}$ and $\underline{x}=\min\{x_g,x_h\}$. Let $u\ge 0$. Certainly $\bar{g}(u)=g(\bar{x}+u)$ and $\underline{g}(u)=g(\underline{x}-u)$ are both decreasing functions of $u$. Similarly $\bar{h}(u)=h(\bar{x}+u)$ and $\underline{h}(u)=h(\underline{x}-u)$ are both increasing functions of $u$.
Thus on the interval $(-\infty,\underline{x}]$, $f(x)$ is maximized at $\underline{x}$ and on the interval $[\bar{x},\infty)$, $f(x)$ is maximized at $\bar{x}$. If we are looking for the maximum value of $f$, we need not consider any points in $(-\infty,\underline{x})\cup(\bar{x},\infty)$; in other words $f$ attains its maximum value somewhere on the interval $[\underline{x},\bar{x}]$.
\end{proof}

We now proceed to show that if Player II chooses pure strategy $\bb^*=(x_2^*,x_2^*)$ and I deviates from $\ba^*=(-x_2^*,-x_2^*)$ to any other pure strategy $(x_1,y_1)$ then it will simply result in a positive expected payoff. 

\begin{lemma}
Suppose $\rho>0$, and $\sigma_x<\sigma_y$. If Player II plays pure strategy $\bb^*=(x_2^*,x_2^*)$ then the only pure strategy on the circle $x_1^2+y_1^2=2x_2^{*2}$ where $K(\ba,\bb^*)\le 0$ is $\ba=(-x_2^*,-x_2^*).$ \label{lemL2OnCircle}
\end{lemma}
\begin{proof}
If $x_1^2+y_1^2=2x_2^{*2}$, $z(\ba,\bb^*)=0$, so
\begin{alignat*}{2}
K(\ba,\bb^*) &= 2x_2^*+(x_1+y_1 - 2x_2^*)\Phi(0)\\
&=2x_2^*+(x_1+y_1 - 2x_2^*)\frac{1}{2}\\
&=x_2^*+\frac{x_1+y_1}{2}.
\end{alignat*}
Geometrically we can see that $x_1+y_1$ is minimized on the circle $x_1^2+y_1^2=2x_2^{*2}$ at $(-x_2^*,-x_2^*)$.
\end{proof}

Against Player II's strategy $\bb^*=(x_2^*,x_2^*)$, any pure strategy $\ba=(x_1,y_1)$ may be represented in terms of $r$ and $\theta$ as $(x_2^*+r\cos\theta, x_2^*+r\sin\theta)$. This will greatly facilitate the remaining proofs\footnote{For convenience we will define $t(\theta):= -(\cos\theta+\sin\theta)$ and $\sigma_\theta^2 := \sigmathetasq$. Note that $t(\theta)=-\sqrt{2}\sin(\theta+\frac{\pi}{4})$.}. In this representation, with $t(\theta)=-(\cos\theta+\sin\theta)$, we can rewrite
\begin{equation}
K(\ba, \bb^*) = 2x_2^* + r(\cos\theta+\sin\theta)\Phi(z) = 2x_2^*-rt(\theta)\Phi(z) \label{eqKabrtheta}
\end{equation}
and
\begin{equation}\label{eqzrtheta}
z(r,\theta) = \frac{2x_2^{*2}-(x_2^*+r\cos\theta)^2-(x_2^*+r\sin\theta)^2}{2r\sqrt{\sigmathetasq}}=\frac{2x_2^*t(\theta)-r}{2\sqrt{\sigma_\theta^2}}
\end{equation}

\begin{figure}[h!]
\centering
\includegraphics[scale=0.31]{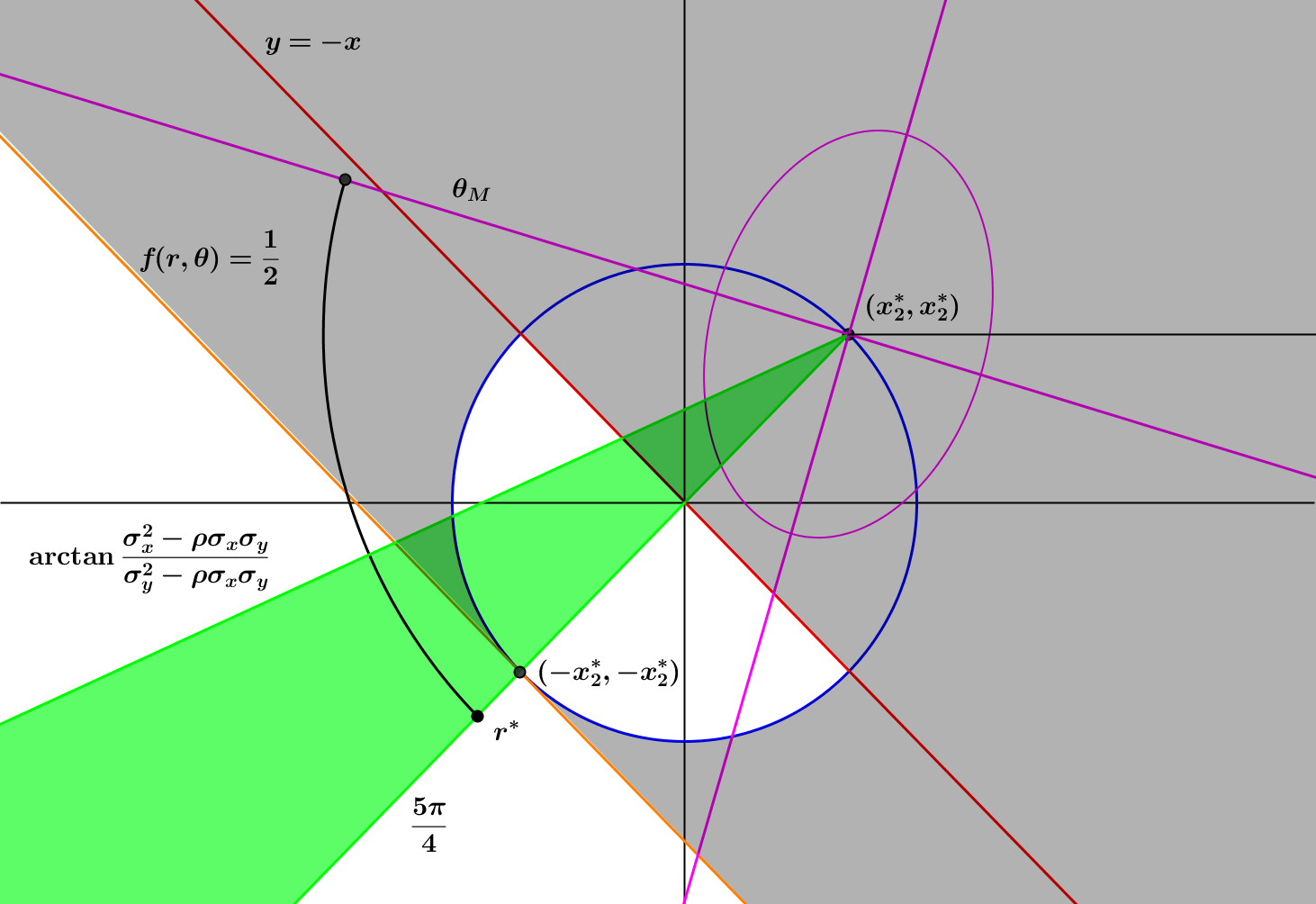}
\caption{A sketch of some features of the problem. With $\bb^*=(x_2^*,x_2^*)$ fixed, the circle in blue are the points where $z=0$, the ellipse in magenta shows the angle of the minor axis of the bivariate distribution (see Lemma~\ref{lemL2GE2}). The two rays in green correspond to Lemma~\ref{lemL2GE1}, and the region between them is where $K(\ba,\bb^*)$ is minimized. From Lemma~\ref{lemL2GE3}, for any point $\ba$ in the region in gray, $K(\ba,\bb^*)>0$. In Lemma~\ref{lemL2GE2} we fix $r^*$ and show that for all points along the arc shown, the expected payoff is positive.}
\end{figure}

\begin{lemma}
Suppose $\rho>0$, and $\sigma_x<\sigma_y$. 
For all pure strategies {$\ba =(x_2^*+r \cos \theta, x_2^*+r\sin\theta)$}, if $K(\ba,\bb^*)$ is minimized then $\theta \in [\arctan\frac{\sigma_x^2-\rho\sigma_x\sigma_y}{\sigma_y^2-\rho\sigma_x\sigma_y} ,  \frac{5\pi}{4}]$. \label{lemL2GE1}
\end{lemma}
\begin{proof}
As in Theorem~\ref{thmL2GE5},  suppose Player I first fixes his net offer $x+y=\tilde{w}<2x_2^*$. 
With $\tilde{w}$ fixed, he now chooses $x$ in order to minimize the the expected payoff $$K(\ba,\bb^*)=2x_2^*+(\tilde{w}-2x_2^*)\Phi(z).$$ Because $\tilde{w}-2x_2^*<0$, he chooses $x$ to maximize $$z=\frac{2x_2^{*2}-x^2-(\tilde{w}-x)^2}{2\sqrt{\sigma_x^2(x_2^*-x)^2+2\rho\sigma_x\sigma_y(x_2^*-x)(x_2^*-\tilde{w}+x)+\sigma_y^2(x_2^*-\tilde{w}+x)^2}}.$$ As in Theorem~\ref{thmL2GE5}, the numerator is maximized when $x=\frac{\tilde{w}}{2}$, which corresponds to $\theta=\frac{5\pi}{4}$. The denominator has a unique minimum when 
$$x
=\frac{x_2^*(\alpha'-\beta')+\beta'\tilde{w}}{\alpha'+\beta'}=\frac{x_2^*(\alpha'-\beta')+\beta'(x+y)}{\alpha'+\beta'},$$ where $\alpha'=\sigma_x^2-\rho\sigma_x\sigma_y$ and $ \beta'=\sigma_y^2-\rho\sigma_x\sigma_y$. 
Solving for $y$, we get $$y=x_2^*\frac{\beta'-\alpha'}{\beta'}+\frac{\alpha'}{\beta'}x.$$
So for all $\tilde{w}$, the set of offers which minimize the denominator is a line with a slope $\frac{\alpha'}{\beta'}$ which passes through $\bb^*$. 
Observe that $\frac{\alpha'}{\beta'}<1$, so $\arctan\frac{\alpha'}{\beta'}<\frac{5\pi}{4}$. Since $\tilde{w}<2x_2^*$ is arbitrary, by Lemma~\ref{lemMaxFracFunction} we know that $K(\ba,\bb^*)$ is minimized for some $(x_2^*+r\cos\theta,x_2^*+r\sin\theta)$ where $\theta\in[\arctan\frac{\alpha'}{\beta'},\frac{5\pi}{4}]$.
\end{proof}

\begin{figure}[h!]
\centering
\includegraphics[scale=0.15]{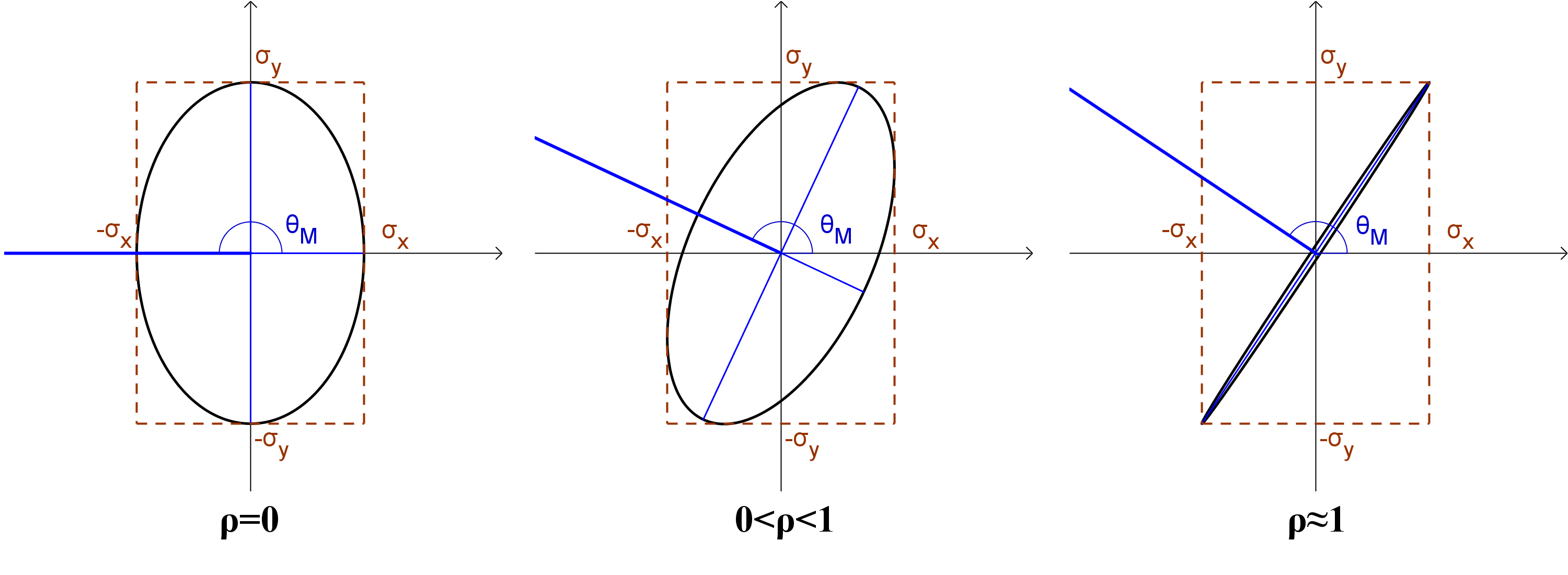}
\caption{An illustration of the angle $\theta_M$ of the minor axis of the bivariate normal distribution for $\sigma_x<\sigma_y$.}
\end{figure}
We will pause briefly to consider the bivariate normal distribution with $\rho>0$ and $\sigma_x<\sigma_y$. In particular, if we look at a contour of constant density, we will have an ellipse. The minor axis of this ellipse will be relevant in Lemma~\ref{lemL2GE2}.
An analysis of the geometry is enough to convince one that the angle $\theta_M>\frac{3\pi}{4}$. Furthermore, this is the direction along which the distribution has the minimum variance.

\begin{lemma}\label{lemthetaM}
Let $\rho>0, \sigma_x<\sigma_y$, $(\xi,\eta)\sim N(\mathbf{0},\mathbf{\Sigma})$, $\Omega=\xi\cos\theta+\eta\sin\theta$ and $\sigma_\theta^2=Var(\Omega)$. Then $\sigma_\theta^2$ is minimized at 
$$\theta_M=\arctan\left(\frac{\sigma_y^2-\sigma_x^2}{2\rho\sigma_x\sigma_y}- \sqrt{\left(\frac{\sigma_y^2-\sigma_x^2}{2\rho\sigma_x\sigma_y}\right)^2+1}\right)>\frac{3\pi}{4}.$$ 
\end{lemma}
It is helpful to realize that $\theta_M$ is the angle of the minor axis of the ellipse formed by any constant-density contour of $f(\xi,\eta)$. The major axis will have the angle $\theta^M$, referenced in the proof.
\begin{proof}
For any angle $\theta$, $Var(\Omega)$ is given by  
\begin{alignat*}{2}
\sigma_\theta^2 &= \sigma_x^2\cos^2\theta + 2\rho\sigma_x\sigma_y\cos\theta\sin\theta + \sigma_y^2\sin^2\theta\\
&=\frac{\sigma_x^2+\sigma_y^2}{2}+\rho\sigma_x\sigma_y\sin2\theta + \frac{\sigma_x^2-\sigma_y^2}{2}\cos2\theta.
\end{alignat*}
This is minimized or maximized when the derivative 
$$2\rho\sigma_x\sigma_y\cos2\theta + (\sigma_y^2-\sigma_x^2)\sin2\theta=0,$$
or equivalently when
$$\tan2\theta = \frac{-2\rho\sigma_x\sigma_y}{\sigma_y^2-\sigma_x^2}.$$
By trigonometric identity this is equivalent to
$$\frac{2\tan\theta}{1-\tan^2\theta}= \frac{-2\rho\sigma_x\sigma_y}{\sigma_y^2-\sigma_x^2},$$
which leads to the quadratic in $\tan \theta$
$$\tan^2\theta - \frac{\sigma_y^2-\sigma_x^2}{\rho\sigma_x\sigma_y}\tan\theta-1=0. $$
This admits solutions
$$\tan\theta = \frac{\sigma_y^2-\sigma_x^2}{2\rho\sigma_x\sigma_y}\pm \sqrt{\left(\frac{\sigma_y^2-\sigma_x^2}{2\rho\sigma_x\sigma_y}\right)^2+1}.$$
Of the two solutions, one is positive while the other is negative. Let $$\tan\theta_M = \frac{\sigma_y^2-\sigma_x^2}{2\rho\sigma_x\sigma_y}- \sqrt{\left(\frac{\sigma_y^2-\sigma_x^2}{2\rho\sigma_x\sigma_y}\right)^2+1}, \text{ and } \tan\theta^M = \frac{\sigma_y^2-\sigma_x^2}{2\rho\sigma_x\sigma_y}+ \sqrt{\left(\frac{\sigma_y^2-\sigma_x^2}{2\rho\sigma_x\sigma_y}\right)^2+1}.
 $$
One can of course verify that $$\tan\theta_M=-\frac{1}{\tan\theta^M}.$$
Since  $\tan\theta^M>1$, $-1<\tan\theta_M<0$. Hence $\sin\theta^M>\cos\theta^M>0$ and $\sin\theta_M<0<\cos\theta_M$. Therefore $\sigma^2_{\theta_M}<\sigma^2_{\theta^M}$, so $\theta_M$ minimizes $\sigma_\theta^2$ while $\theta^M$ maximizes $\sigma_\theta^2$. Finally we note that because $\tan\theta_M > -1$,
$$\theta_M > \arctan (-1)=\frac{3\pi}{4}.$$
\end{proof}

\begin{lemma}
Suppose $\rho>0$, and $\sigma_x<\sigma_y$.
For any pure strategy $\ba =(x_2^*+r \cos \theta, x_2^*+r\sin\theta) \not=-\bb^*=(x_2^*,x_2^*)$ with $\theta \in [\theta_M, \frac{5\pi}{4}]$ and $x_1^2+y_1^2>2x_2^{*2}$, $K(\ba,\bb^*)> 0$. \label{lemL2GE2}
\end{lemma}
\begin{proof}
Observe that $K>0$ is equivalent to  
$$2x_2^*-rt(\theta)\Phi(z)>0\quad\Leftrightarrow\quad\Phi(z) < f(r,\theta)= \frac{2x_2^*}{rt(\theta)}. $$
Let us fix $r^*> 2\sqrt{2}x_2^*$. The proof proceeds via three claims:\\
{\bf Claim 1:} $f({r^*},\theta)$ is a decreasing function for $\theta \in [\theta_M, \frac{5\pi}{4}]$.\\
{\bf Proof of Claim 1:} 
This claim follows immediately after noting that $t(\theta)$ is an increasing function on $(\frac{3\pi}{4},\frac{5\pi}{4})$, and, from Lemma~\ref{lemthetaM}, $\frac{3\pi}{4}<\theta_M$. \\
{\bf Claim 2:} $\Phi(z({r^*},\theta))$ is an increasing function for $\theta \in [\theta_M, \frac{5\pi}{4}]$.\\
{\bf Proof of Claim 2:} 
Recall that $ z = \frac{2x_2^*t(\theta)-r}{2\sqrt{\sigma_\theta^2}}.$
For $\theta \in [\frac{3\pi}{4},\frac{5\pi}{4}]$, $r_0=2x_2^*t(\theta)$ is increasing and attains its maximum value of $2\sqrt{2}x_2^*$ when $\theta=\frac{5\pi}{4}$. 
Since $r^*> 2\sqrt{2}x_2^* \ge r_0$, ${z(r^*,\theta)<0}$. Because $\sigma_\theta^2$ attains its minimum at $\theta_M$ and is maximized at $\theta^M>\frac{5\pi}{4}$, 
$\sigma_\theta^2$ is increasing on $[\theta_M, \frac{5\pi}{4}]$.

Consider $$|z|=\frac{r^*-2x_2^*t(\theta)}{2\sqrt{\sigma_\theta^2}}.$$ For $\theta\in[\theta_M,\frac{5\pi}{4}]$, $t(\theta)$ increases so the numerator is decreasing. Meanwhile the denominator is increasing. 
Thus $\Phi(z(r^*,\theta))$ is an increasing function in $[\theta_M,\frac{5\pi}{4}] $.\\
{\bf Claim 3:} $\Phi(z(r^*,\frac{5\pi}{4})) < f(r^*,\frac{5\pi}{4})$.\\
{\bf Proof of Claim 3:} If we fix $\theta = \frac{5\pi}{4}$, then the players are in the one-dimensional FOA game, and we already know that $\ba^*=-\bb^*$ (i.e. $r^*=2\sqrt{2}x_2^*$) is the globally optimal strategy for Player I to play against $\bb^*$. Since we have fixed $r^*>2\sqrt{2}x_2^*$, Player~I is not playing optimally, so $K>0$ which is equivalent to the claim. \\
From these three claims it follows that $K>0$ for $r>2\sqrt{2}x_2^*$ and $\theta\in[\theta_M,\frac{5\pi}{4}]$.
\end{proof}

\begin{lemma}
Suppose $\rho>0$, and $\sigma_x<\sigma_y$. Then  $$\theta_M<\arctan\frac{\alpha'}{\beta'},$$ where $\alpha'=\sigma_x^2-\rho\sigma_x\sigma_y$ and $\beta'=\sigma_y^2-\rho\sigma_x\sigma_y$.\label{L2intervalsubset}
\end{lemma}
\begin{proof}
We just need to show that the slope $\frac{\alpha'}{\beta'}$ is greater than $\tan\theta_M$. In other words, we must show that
$$\frac{\sigma_y^2-\sigma_x^2}{2\rho\sigma_x\sigma_y} - \sqrt{\left(\frac{\sigma_y^2-\sigma_x^2}{2\rho\sigma_x\sigma_y}\right)^2+1} < \frac{\sigma_x^2-\rho\sigma_x\sigma_y}{\sigma_y^2-\rho\sigma_x\sigma_y}.$$
Suppose for some $\sigma_x<\sigma_y$ and $0<\rho<1$ we have a contradiction, that is
$$ \frac{\sigma_y^2-\sigma_x^2}{2\rho\sigma_x\sigma_y} - \sqrt{\left(\frac{\sigma_y^2-\sigma_x^2}{2\rho\sigma_x\sigma_y}\right)^2+1} \ge \frac{\sigma_x^2-\rho\sigma_x\sigma_y}{\sigma_y^2-\rho\sigma_x\sigma_y}. $$
This inequality is equivalent to the following inequalities:
\begin{alignat*}{2}
\frac{\beta'-\alpha'}{2\rho\sigma_x\sigma_y} -\sqrt{\left(\frac{\beta'-\alpha'}{2\rho\sigma_x\sigma_y}\right)^2+1} &\ge \frac{\alpha'}{\beta'}\\ 
\frac{\beta'-\alpha'}{2\rho\sigma_x\sigma_y} - \frac{\alpha'}{\beta'}&\ge  \sqrt{\left(\frac{\beta'-\alpha'}{2\rho\sigma_x\sigma_y}\right)^2+1}\\
\left(\frac{\beta'-\alpha'}{2\rho\sigma_x\sigma_y}\right)^2 -\frac{(\beta'-\alpha')\alpha'}{\rho\sigma_x\sigma_y\beta'} + \frac{\alpha'^2}{\beta'^2} &\ge\left(\frac{\beta'-\alpha'}{2\rho\sigma_x\sigma_y}\right)^2+1\\
-\frac{(\beta'-\alpha')\alpha'}{\rho\sigma_x\sigma_y\beta'} + \frac{\alpha'^2}{\beta'^2} -1 &\ge 0\\
-(\beta'-\alpha')\alpha'\beta' + \alpha'^2\rho\sigma_x\sigma_y-\beta'^2\rho\sigma_x\sigma_y&\ge 0\\
(\alpha'-\beta')\alpha'\beta' +(\alpha'-\beta')(\alpha'+\beta')\rho\sigma_x\sigma_y&\ge 0\\
\alpha'\beta' + (\alpha'+\beta')\rho\sigma_x\sigma_y&\le  0\\
(\sigma_x^2\sigma_y^2-\rho\sigma_x^3\sigma_y-\rho\sigma_x\sigma_y^3+\rho^2\sigma_x^2\sigma_y^2) + (\sigma_x^2-2\rho\sigma_x\sigma_y+\sigma_y^2)\rho\sigma_x\sigma_y&\le 0\\
\sigma_x^2\sigma_y^2 - \rho^2\sigma_x^2\sigma_y^2&\le 0\\
(1-\rho^2)\sigma_x^2\sigma_y^2&\le 0.
\end{alignat*}
This is of course impossible as $-1<\rho<1$. Note that we used the facts that $\beta'=\sigma_y^2-\rho\sigma_x\sigma_y>0$ and $\alpha'-\beta'=\sigma_x^2-\sigma_y^2<0$.
\end{proof}

\begin{corollary}
If $\rho>0$, $\sigma_x<\sigma_y$, and $\bb^*=(x_2^*,x_2^*)$ then for all  $\ba=(x_1,y_1)$ with $x_1^2+y_1^2>2x_2^{*2}$, $K(\ba,\bb^*)> 0$.\label{corL2OutCircle}
\end{corollary}
\begin{proof}
By Lemma~\ref{lemL2GE1}, $K(\ba,\bb^*)$ is minimized for some $\theta\in[\arctan\frac{\sigma_x^2-\rho\sigma_x\sigma_y}{\sigma_y^2-\rho\sigma_x\sigma_y},\frac{5\pi}{4}]$. By Lemma~\ref{L2intervalsubset}, $[\arctan\frac{\sigma_x^2-\rho\sigma_x\sigma_y}{\sigma_y^2-\rho\sigma_x\sigma_y},\frac{5\pi}{4}]\subset[\theta_M,\frac{5\pi}{4}]$, and by Lemma~\ref{lemL2GE2}, $K(\ba,\bb^*)>0 $ for any $\theta\in [\theta_M,\frac{5\pi}{4}]$ with $x_1^2+y_1^2>2x_2^{*2}$. 
\end{proof}

Now that we have shown that against $(x_2^*,x_2^*)$ all pure strategies for Player I outside the circle $x^2+y^2=2x_2^{*2}$ will give a positive expected payoff, we consider strategies within the circle.

\begin{lemma}
Suppose $\rho>0$, $\sigma_x<\sigma_y$ and $\bb^*=(x_2^*,x_2^*)$.
For all pure strategies $\ba=(x_1,x_2)$
such that $x_1^2+y_1^2<2x_2^{*2}$,  $K(\ba,\bb^*)> 0$. \label{lemL2InCircle}
\end{lemma}
The proof relies on the concavity of the CDF of the normal distribution within the circle in question.
\begin{proof}
From Lemma~\ref{lemL2GE3}, we need only show that $K(\ba,\bb^*)>0$ for all $\ba$ in the semi-circle described by
$$\begin{cases}
x+y<0,\\ 
x^2+y^2<{2}x_2^{*2}.
\end{cases}$$ 
In terms of $\theta$, we are restricting our attention to $\theta\in(\pi,\frac{3\pi}{2})$. For the angles $\theta$ in question, $t(\theta)>1$. Recall from (\ref{eqKabrtheta}) that $K(\ba, \bb^*)>0$ is equivalent to 
$$\Phi(z) < f(r,\theta)=\frac{2x_2^*}{rt(\theta)}. $$

First we fix $\tilde{\theta}\in (\pi,\frac{3\pi}{2})$. Let $r_0=2x_2^*t(\tilde{\theta})$. Note by definition that $z(r_0,\tilde{\theta})=0$. Since $z=\frac{r_0-r}{2\sqrt{\sigma_\theta^2}}$,  it is straightforward to show that
$$\left.\frac{d}{dr}\Phi(z)\right|_{z=0}
=\left.\phi(z)\frac{dz}{dr}\right|_{z=0}
=\frac{1}{\sqrt{2\pi}}\frac{-1}{2\sqrt{\sigma_{\tilde{\theta}}^2}}
=\frac{-1}{2\sqrt{2\pi\sigma_{\tilde{\theta}}^2}}.$$
Define $y$ as the line tangent to $\Phi$ at $(r_0,\frac{1}{2})$, specifically,
$$y(r,{\tilde{\theta}})=-\frac{r-r_0}{2\sqrt{2\pi\sigma_{\tilde{\theta}}^2}}+\frac{1}{2}.$$

Note $\Phi$ is a concave function for $r<r_0$. Therefore, 
$\Phi(z(r,{\tilde{\theta}}))\le y(r,{\tilde{\theta}})$.
To demonstrate that $f>\Phi$ for all $r<r_0$, it suffices to show that $f>y$ for all $r$.
Since $f$ and $y$ are both continuous functions and $\lim_{r\to 0^+}f(r,{\tilde{\theta}}) =\infty \gg y(0,{\tilde{\theta}})$, it suffices to show that  $f\not=y$ for any $r$. 
If the two curves do intersect, then there is at least one solution to the equation 
$$\frac{\sqrt{2\pi(\alpha+\beta)}}{2rt({\tilde{\theta}})} =-\frac{1}{2\sqrt{2\pi\sigma_{\tilde{\theta}}^2}}r +\frac{t({\tilde{\theta}})\sqrt{\alpha+\beta}}{4\sqrt{\sigma_{\tilde{\theta}}^2}} + \frac{1}{2},$$
or equivalently 
\begin{alignat*}{2}
0 &= \frac{1}{2\sqrt{2\pi\sigma_{\tilde{\theta}}^2}}r^2 -\left(\frac{t({\tilde{\theta}})\sqrt{\alpha+\beta}}{4\sqrt{\sigma_{\tilde{\theta}}^2}} + \frac{1}{2}\right)r+\frac{\sqrt{2\pi(\alpha+\beta)}}{2t({\tilde{\theta}})}\\
&=r^2 -\left(\frac{t({\tilde{\theta}})\sqrt{2\pi(\alpha+\beta)}}{2} + \sqrt{2\pi\sigma_{\tilde{\theta}}^2} \right)r+\frac{2\pi\sqrt{(\alpha+\beta)\sigma_{\tilde{\theta}}^2}}{t({\tilde{\theta}})}.
\end{alignat*}
We have a quadratic in $r$. Let $$\hat{r}=\frac{t({\tilde{\theta}})\sqrt{2\pi(\alpha+\beta)}}{4} + \frac{\sqrt{2\pi\sigma_{\tilde{\theta}}^2}}{2}$$
and
$$\Delta = \left(\frac{t({\tilde{\theta}})\sqrt{2\pi(\alpha+\beta)}}{2} + \sqrt{2\pi\sigma_{\tilde{\theta}}^2} \right)^2-\frac{8\pi\sqrt{(\alpha+\beta)\sigma_{\tilde{\theta}}^2}}{t({\tilde{\theta}})}$$

If $\Delta<0$ then we are done. Let us assume that $\Delta\ge0$. If $f(r^*,\tilde{\theta})=y(r^*,\tilde{\theta})$, it must that $r^* < r_0$; for $r \ge r_0$, $f(r,\tilde{\theta}) >\Phi(z(r^*,\tilde{\theta}))>y(r^*,\tilde{\theta})$. This gives us a condition, namely $\hat{r}+\frac{\sqrt{\Delta}}{2}< r_0$. 
\begin{alignat*}{2}
\hat{r}+\frac{\sqrt{\Delta}}{2} &< r_0\\
\frac{1}{\sqrt{2\pi}}\left(\hat{r}+\frac{\sqrt{\Delta}}{2}\right) &< \frac{r_0}{\sqrt{2\pi}}\\
\frac{t(\tilde{\theta})\sqrt{\alpha+\beta}}{4} + \frac{\sqrt{\sigma_{\tilde{\theta}}^2}}{2}+ \frac{\sqrt{\left(\frac{t(\tilde{\theta})\sqrt{\alpha+\beta}}{2} + \sqrt{\sigma_{\tilde{\theta}}^2} \right)^2-\frac{2\sqrt{(\alpha+\beta)\sigma_{\tilde{\theta}}^2}}{t(\tilde{\theta})} }}{2}&< \frac{t(\tilde{\theta})\sqrt{\alpha+\beta}}{2}\\
{\sqrt{\sigma_{\tilde{\theta}}^2}}+ \sqrt{\left(\frac{t(\tilde{\theta})\sqrt{\alpha+\beta}}{2} + \sqrt{\sigma_{\tilde{\theta}}^2} \right)^2-\frac{2\sqrt{(\alpha+\beta)\sigma_{\tilde{\theta}}^2}}{t(\tilde{\theta})} }&< \frac{t(\tilde{\theta})\sqrt{\alpha+\beta}}{2}\\
\sqrt{\left(\frac{t(\tilde{\theta})\sqrt{\alpha+\beta}}{2} + \sqrt{\sigma_{\tilde{\theta}}^2} \right)^2-\frac{2\sqrt{(\alpha+\beta)\sigma_{\tilde{\theta}}^2}}{t(\tilde{\theta})} }&< \frac{t(\theta)\sqrt{\alpha+\beta}}{2}-\sqrt{\sigma_{\tilde{\theta}}^2}\\
\left(\frac{t({\tilde{\theta}})\sqrt{\alpha+\beta}}{2} + \sqrt{\sigma_{\tilde{\theta}}^2} \right)^2-\frac{2\sqrt{(\alpha+\beta)\sigma_{\tilde{\theta}}^2}}{t({\tilde{\theta}})} &< \left(\frac{t({\tilde{\theta}})\sqrt{\alpha+\beta}}{2} - \sqrt{\sigma_{\tilde{\theta}}^2} \right)^2\\
\left(\frac{t({\tilde{\theta}})\sqrt{\alpha+\beta}}{2} + \sqrt{\sigma_{\tilde{\theta}}^2} \right)^2-\left(\frac{t({\tilde{\theta}})\sqrt{\alpha+\beta}}{2} - \sqrt{\sigma_{\tilde{\theta}}^2} \right)^2 &< \frac{2\sqrt{(\alpha+\beta)\sigma_{\tilde{\theta}}^2}}{t({\tilde{\theta}})}\\
2t({\tilde{\theta}})\sqrt{(\alpha+\beta)\sigma_{\tilde{\theta}}^2} &< \frac{2\sqrt{(\alpha+\beta)\sigma_\theta^2}}{t(\theta)}\\
t({\tilde{\theta}}) &< \frac{1}{t({\tilde{\theta}})}
\end{alignat*}
in other words, $t({\tilde{\theta}})<1$, which is a contradiction.
\end{proof}
The following is the main result.
\begin{theorem}
For $\rho>0$, if the arbiter uses $L_2$ distance as a decision criterion, then $\ba^*=(-x_2^*,-x_2^*), \bb^*=(x_2^*,x_2^*)$ is a pure global equilibrium pair.
\end{theorem}
\begin{proof}
This follows from the previous lemmas. WLOG $\sigma_x \le \sigma_y$. If Player II plays pure strategy $\bb^*$, then for any pure strategy $\ba=(x_1,y_1)$, $K(\ba,\bb^*)\ge 0$, and equality is only achieved when $\ba=\ba^*$. 
\end{proof}

\section{Variability of Issue-by-Issue and Whole Package Outcomes}
Having shown that under an $L_2$ distance criterion there is a unique pure optimal strategy pair, we consider the question of whether the issue-by-issue or whole-package variant is more in line with the aims of FOA. Since FOA makes arbitration a costly alternative by its inherent uncertainty, we may compare the uncertainty (i.e. variance) between equilibrium strategies under the two mechanisms. It may come as no surprise that the arbitrated outcome in WPFOA has a higher variance.
\begin{theorem}
Under $L_2$ criterion, the expected payoff is zero under either Issue-by-Issue rules or Whole-Package. If both players choose optimal strategies then the variances of the awards, respectively, are $\frac{\pi}{2}(\sigma_x^2+\sigma_y^2)$ and $\frac{\pi}{2}(\sigma_x^2+2\rho\sigma_x\sigma_y+\sigma_y^2)$.
\end{theorem}
\begin{proof}
Under IBIFOA, since the components are awarded independently, the variance is 
\begin{alignat*}{2}Var(K) &=Var(K(x)+K(y))\\
&=E(K(x)^2)+E(K(y)^2)\\
&=\frac{1}{2}\left(2\frac{2\pi\sigma_x^2}{4}\right)+\frac{1}{2}\left(2\frac{2\pi\sigma_y^2}{4}\right)\\
&=\frac{\pi}{2}(\sigma_x^2+\sigma_y^2)
\end{alignat*}
Under WPFOA the variance is
\begin{alignat*}{2}Var(K)
&=E(K^2)\\
&=\frac{1}{2}(2x_1^*)^2+\frac{1}{2}(2x_2^*)^2\\
&=4x_2^{*2}\\
&=\frac{\pi}{2}(\sigma_x^2+2\rho\sigma_x\sigma_y+\sigma_y^2)
\end{alignat*}
\end{proof}
Thus we argue that quantitative issues should be arbitrated by package rather than independently to provide a stronger motivation to the parties to reach ``security in agreement'' \citep{stevens1966}. 

\section{Conclusions}
We have developed a model of two issue final-offer arbitration as a zero-sum game where both players are risk-neutral, issues under dispute are quantitative and the values are additive, the arbiter chooses a fair settlement from a bivariate normal distribution commonly known to both players and measures how `reasonable' a final-offer is by its $L_2$ distance from this reasonable settlement. We have shown, among other properties, that with reasonable assumptions the game has a value of zero. If the two components are not too negatively correlated, locally optimal pure strategies are derived. If we further assume that the issues are positively correlated, these represent the unique optimal strategy pair. Finally it was observed that in this case whole-package FOA leads to an outcome with greater variance than IBI, and would act as a greater motivator to reach agreement in negotiations.

This represents only an initial model of the multi-issue FOA game. Many variants are worthy of consideration. Firstly the arbiter may use one of any number of decision criteria including $L_1$ distance, $L_\infty$ distance, total absolute difference, and Mahalanobis distance. It may be the case that the final-offer vectors must be standardized before measuring distance, and Players valuation of a final-offer may be more complicated than the sum of the two components. Another obvious extension is to look at the $n$-issue game. 

Finally it worth considering an extension of final-offer arbitration to $n$-player games. This would have applications for inheritance splitting, for example, where the heirs cannot agree on a fair split and need to bring in an arbiter. To our knowledge, final-offer arbitration has not been used in this scenario but we feel it would be an effective means to encourage agreement among the players.

\section{Acknowledgments}
My sincere thanks to T.E.S. Raghavan for first suggesting this problem, for his keen interest,  encouragement, guidance and for many hours of discussion and review.

\bibliography{foa.bib}

\end{document}